\documentclass[11pt]{amsart}
\usepackage{appendix}
\usepackage{graphicx}
\usepackage{chngpage}
\usepackage{multirow}
\usepackage{hyperref}
\usepackage[all,cmtip]{xy}
\usepackage{amssymb}
\usepackage{amsmath}
\usepackage{longtable}

\usepackage{pstricks}
\usepackage{color}
\usepackage{fancyhdr}
\theoremstyle{plain}
\newtheorem{thm}{Theorem}[section]
\newtheorem{theorem}[thm]{Theorem}

\newtheorem{lemma}[thm]{Lemma}

\theoremstyle{definition}
\newtheorem{remark}[thm]{Remark}

\newtheorem{notation}[thm]{Notation}
\newtheorem{definition}[thm]{Definition}

\newtheorem{example}[thm]{Example}

\numberwithin{equation}{section}
\newcommand{\wt}{\widetilde}
\newcommand{\ti}{\times}

\newcommand{\rt}{\rtimes}

\newcommand{\la}{{\langle}}
\newcommand{\ra}{{\rangle}}
\newcommand{\al}{\alpha}

\newcommand{\Aut}{{\rm Aut}}

\newcommand{\Tr}{{\rm tr}}
\newcommand{\Diag}{{\rm diag}}

\newcommand{\Ord}{{\rm ord}}

\newcommand{\GL}{{\rm GL}}
\newcommand{\PGL}{{\rm PGL}}

\newcommand{\Det}{{\rm det}}
\newcommand{\SL}{{\rm SL}}

\newcommand{\C}{{\mathbb C}}

\renewcommand{\P}{{\mathbb P}}
\newcommand{\Q}{{\mathbb Q}}

\begin{document}

% \title{Mccorrespondence and new Calabi-Yau threefolds}
 
 \title{McKay correspondence and New Calabi-Yau threefolds}
 
 \author{Xun Yu} 
\address{Center for Geometry and its Applications, POSTECH, Pohang, 790-784, Korea}
\email{yxn100135@postech.ac.kr}

%\thanks{Mathematics Subject Classification (2010). Primary 14H45; Secondary 14H50, 14J32, 14N25}

\begin{abstract} In this note, we consider crepant resolutions of the quotient varieties of smooth quintic threefolds by Gorenstein group actions. We compute their Hodge numbers via McKay correspondence. In this way, we find some new pairs $(h^{1,1},\;h^{2,1})$ of Hodge numbers of Calabi-Yau threefolds.
\end{abstract}

\maketitle

\fancyhf{}
\renewcommand{\headrulewidth}{0pt}
\fancyhead[LE,RO]{\thepage}
\fancyhead[CE]{Xun Yu }
\fancyhead[CO]{McKay correspondence and new Calabi-Yau threefolds}

\section{Introduction}

Calabi-Yau threefolds are important and interesting objects both in algebraic geometry (e.g., classification theory ) and physics (e.g., mirror symmetry), and they have been extensively studied by many people in the last few decades. Unlike lower dimensional Calabi-Yau manifolds (elliptic curves, K3 surfaces) classification of topological types of Calabi-Yau threefolds turns out to be very hard (cf. \cite{Re02}). In fact, even boundedness of their Hodge numbers $h^{1,1}$ and $h^{2,1}$  is still an open problem. Thus, constructing Calabi-Yau threefolds with ``new'' pairs of Hodge numbers ($h^{1,1},\;h^{2,1}$) is of interest. 

Let $X$ be a smooth quintic threefold, a most basic example of Calabi-Yau threefolds, and let $G$ be a Gorenstein  subgroup of the automorphism group $\Aut(X)$ of $X$. Recently both $\Aut(X)$ and $G$ are essentially classified (\cite{OY15}). By \cite[Theorem 1.2]{BKR01} the quotient variety $X/G$ always has a crepant resolution, say $\widehat{X/G}$, which is again a Calabi-Yau threefold.  The aim of this note, is to compute the Hodge numbers of $\widehat{X/G}$ for various pairs $(X,G)$. Our main result is Theorem \ref{thm:main} (see also Remark \ref{rm:main}). In particular, there are $43$ pairs of integers appearing as $(h^{1,1}(\widehat{X/G}), h^{2,1}(\widehat{X/G}))$, and some of them represent ``new'' Calabi-Yau threefolds. In the literature, the resolutions $\widehat{X/G}$ for some pairs $(X,G)$ have been studied (for examples, \cite{GP90}, \cite{ALR90}, \cite{CDGP91},  \cite{Za93}, \cite{BD96}, \cite{BGK12}, \cite{St12}, \cite{BG14}). The proof of Theorem \ref{thm:main} is based on Table \ref{tab:main}, which gives us many examples of pairs $(X,G)$.  Then we use McKay correspondence (Theorem \ref{thm:mckay}, formulas (\ref{eq:h(X,G)}), (\ref{eq:e(X,G)2}), Tables \ref{tab:g}, \ref{tab:pairs}) to compute these numbers. 

In general, Mirror symmetry predicts that Calabi-Yau threefolds  appear in pairs. For example, a natural mirror family of quintic Calabi-Yau threefolds is given by a crepant resolution $\widehat{X_{\lambda}/G}$ of the Gorenstein  quotient of 
$$X_{\lambda} = (x_1^5 + x_2^5 + x_3^5 + x_4^5 + x_5^5 - 5\lambda x_1x_2x_3x_4x_5 = 0)$$
by $G\cong\mu_5^3 \simeq \mu_5^4/\mu_5$. The  manifolds $\widehat{X_{\lambda}/G}$ ($\lambda^5 \not= 1$) are smooth Calabi-Yau threefolds of 
$$h^{1,1}(\widehat{X_{\lambda}/G}) = 101 = h^{1,2}(X_{\lambda})\,\, ,\,\, h^{1,2}(\widehat{X_{\lambda}/G}) = 1 = h^{1,1}(X_{\lambda})\,\, .$$ Note that the ``mirror partners'' of $27$ pairs in the list in Theorem \ref{thm:main} are in the same list and the mirror partners of the remaining $16$ pairs are not in the list but are already found by other methods (cf. \cite{Ju}).

Many Calabi-Yau threefolds (especially, those of small Picard ranks \cite{LOP13}, \cite{Og14}) can be acted on by finite groups. However, classification of such finite groups is still widely open. In general, if a finite group acts on an algebraic variety, the fixed loci of elements of the group or the relationship between the group and the quotient variety may be very useful to control the group (cf. \cite{Mu88}, \cite{Xi96}, \cite{HMX13}). We hope that the computational data (Tables \ref{tab:g}, \ref{tab:pairs}, \ref{tab:main}) in this note are helpful for future study of finite groups of automorphisms of Calabi-Yau threefolds. 

\medskip
 \noindent
{\bf Notations and conventions.}  We use the following notations to describe groups.\\[.2cm]If $A\in \GL(n,\mathbb{C})$, then we use $[A]$ denote the corresponding element in $\PGL(n, \mathbb{C})$. 

  $I_n:=$ the identity matrix of rank $n$;
  
  $\xi_k:=e^{\frac{2\pi i}{k}}$ a $k$-th primitive root of unity, where $k$ is a positive integer;
  
%   $\mathcal{S}(n):=$ the polynomial ring $\mathbb{C}[x_1,\cdots,x_n]$;
%   
%    $\mathcal{S}(n,d):=$ the set of homogeneous polynomials of degree $d$ in $\mathcal{S}(n)$.

    If $B_1,...,B_k$ are square matrices, then we use $\Diag(B_1,...,B_k)$ to denote the obvious block diagonal matrix.
    
     We use $\pi: \GL(n,\mathbb{C})\longrightarrow \PGL(n,\mathbb{C})$ to denote the natural quotient map.
     
     Let $G$ be a finite group and $p$ be a prime. If no confusion causes, we use $G_p$ to denote a Sylow $p$-subgroup of $G$.
     
    The following is the list of symbols of finite groups used in this article:
     
     $C_n$: a cyclic group of order $n$,
     
     $D_{2n}$: a dihedral group of order $2n$,
     
     $S_n(A_n)$: a symmetric (alternative) group of degree $n$,
     
     $Q_8$: a quaternion group of order 8.\\[.5cm]

\section{McKay correspondence and Hodge numbers of crepant resolutions}\label{ss:mckay}

Let $X$ be a Calabi-Yau threefold (i.e., a three dimensional complex smooth projective variety with trivial canonical bundle and vanishing first Betti number). Let $G$ be a finite {\it Gorenstein } (i.e., $G$ fixes a non zero holomorphic 3-form on $X$) subgroup of the automorphism group $\Aut(X)$ of $X$. The quotient variety $X/G$ has only Gorenstein singularities and it always has a crepant resolution (\cite[Theorem 1.2]{BKR01}), say $\widehat{X/G}$. The variety $\widehat{X/G}$ is again a Calabi-Yau threefold and its Hodge numbers can be computed by the classical McKay correspondences due to \cite{IR96}, \cite{BD96}. In order to discuss such computation in details, we would like to introduce some notations  first (mostly following \cite{BD96}). For any $g\in G$, we set $X^g:=\{x\in X \;| \;g(x)=x\}.$ Let $C(g):=\{h\in G \;| \;hg=gh\}$. Then the action of $C(g)$ on $X$ can be restricted on $X^g$. For any point $x\in X^g$, the eigenvalues of $g$ in the holomorphic tangent space $T_x$ are roots of unity:

$$e^{2\pi \sqrt{-1}\al_1},\;e^{2\pi \sqrt{-1}\al_2},\;e^{2\pi \sqrt{-1}\al_3}$$ where $0\leq \al_j <1\; (j=1,2,3)$ are locally constant functions on $X^g$ with values in $\Q$. The fixed locus $X^g$ is  always a disjoint union of finitely many {\it smooth} subvarieties of $X$ (cf.\cite[Lemma 4.1]{Do69} ). We write $X^g=X_1(g)\cup\cdots \cup X_{r_g}(g)$, where $X_1(g),...,X_{r_g}(g)$ are the smooth connected components of $X^g$. For each $i\in \{1,...,r_g\}$, the {\it fermion shift number} $F_i(g)$ is defined to be equal to the value of $\al_1+\al_2+\al_3$ on the connected component $X_i(g)$. Since $X$ is three dimensional, it follows that $F_i(g)\in \{0,1,2\}$. Notice that it could happen that some element in $C(g)$ permutes the connected components $X_i(g)$ (cf. the proof of Theorem \ref{thm:main}). For $j\in \{0,1,2\}$, we set $X_j^g:=\cup_{F_i(g)=j}X_{i}(g)$. We denote by $h^{p,\;q}_{C(g)}(X_j^g)$ the dimension of the subspace of $C(g)$-invariant elements in $H^{p,\;q}(X_j^g)$. We set
\begin{equation}\label{eq:hg}
  h_g^{p,\;q}(X,G):=h_{C(g)}^{p,\;q}(X_0^g)+h_{C(g)}^{p-1,\;q-1}(X_1^g)+h_{C(g)}^{p-2,\;q-2}(X_2^g).
\end{equation}

The {\it orbifold Hodge numbers} of $X/G$ are defined by 
\begin{equation}\label{eq:h(X,G)}
h^{p,\;q}(X,\;G):=\sum_{\{g\}}h_g^{p,\;q}(X,\;G)
\end{equation}
where $\{g\}$ runs over the conjugacy classes of $G$, so that $g$ represents $\{g\}$.

The {\it orbifold Euler number} of $X/G$ is defined by 
\begin{equation}\label{eq:e(X,G)}
e(X,\; G):=\frac{1}{|G|}\sum_{gh=hg}e(X^g\cap X^h)
\end{equation}
  or, equivalently,  
  \begin{equation}\label{eq:e(X,G)2}
  e(X,\; G):=\frac{1}{|G|}\sum_{\{g\}}\#(\{g\})\sum_{h\in C(g)} e(X^g\cap X^h),
  \end{equation}
    where $\#(\{g\})$ means the size of the conjugacy class represented by $g$.

\begin{theorem}\label{thm:mckay}
(\cite[Theorem 6.14 and Corollary 6.15]{BD96}) The Euler characteristic of  $\widehat{X/G}$ is equal to the orbifold Euler number of $X/G$, i.e., $$e(\widehat{X/G})=e(X,\; G).$$ Moreover, $$h^{p,\;q}(\widehat{X/G})=h^{p,\;q}(X,\;G).$$
\end{theorem}

%
%\begin{proposition}\label{pp:formula}
%The Hodge numbers $h^{1,\;1}(\widehat{X/G})$ and $h^{2,\;1}(\widehat{X/G})$ are computed by $$h^{1,\;1}(\widehat{X/G})=h_{G}^{1,\;1}(X)+\sum_{\{g\}\neq \{1_G\}}h_{C(g)}^{0,\;0}(X_1^g) \;$$ and $$h^{2,\;1}(\widehat{X/G})=(h_{G}^{1,\;1}(X)+\sum_{\{g\}\neq \{1_G\}}h_{C(g)}^{0,\;0}(X_1^g))-\frac{1}{2}e(X,G)$$ where $h_{G}^{1,\;1}(X)$ is the dimension of the subspace of $G$-invariant elements in $H^{1,\;1}(X)$ and $1_G$ is the  identity element of $G$.
%\end{proposition}
%
%\begin{proof}
%Let $g\in G$. If $g=1_G$, then $X=X^g=X_0^g$ and $h_g^{1,\;1}(X,G)=h_{G}^{1,\;1}(X)$. If $g\neq 1_G$, then $X_0^g=\emptyset$ and $h_g^{1,\;1}(X,G)=h_{C(g)}^{0,\;0}(X_1^g)$. Then this proposition is just an immediate consequence of Theorem \ref{thm:mckay}.
%
%
%\end{proof}

The fundamental group of $\widehat{X/G}$ is well-understood (at least, when $\pi_1(X)$ is trivial):

\begin{theorem}\label{thm:pi1}
Suppose $X$ is simply connected. Let $H$ be the normal subgroup of $G$ generated by those elements which have fixed points. Then $$\pi_1(\widehat{X/G})\cong \pi_1(X/G)\cong G/H.$$
\end{theorem}

\begin{proof}
By \cite{Ar68}, $\pi_1(X/G)\cong G/H$. By \cite[Theorem 7.8]{Ko93}, $\pi_1(\widehat{X/G})\cong \pi_1(X/G)$. Therefore, the theorem is proved.
\end{proof}

\section{Gorenstein actions on quintics and new Calabi-Yau threefolds}

Let $X\subset\P^4$ be a smooth quintic threefold defined by a degree five homogeneous polynomial $F=F(x_1,...,x_5)$. By a result of Matsumura-Monsky (\cite{MM63}), 
$${\rm Aut}\, (X) = \{\varphi \in \PGL(5, \C) | \varphi(X) = X\}\,\, ,$$
 ${\rm Aut}\, (X)$ is always a {\it finite} group. Let $G< \Aut(X)$ be a Gorenstein subgroup. As abstract groups, both $\Aut(X)$ and $G$ are essentially classified by \cite[Theorems 2.2 and 10.4]{OY15}. Like in Section \ref{ss:mckay}, we denote by $\widehat{X/G}$  a crepant resolution of the quotient variety $X/G$. In this section, we will compute the Hodge numbers and the fundamental group of $\widehat{X/G}$ for many pairs $(X,G)$.

   \begin{notation} 
   Let $A=(a_{ij})\in \GL(5,\mathbb{C})$. We denote by $A(F)$ the homogeneous polynomial $$F(\sum_{i=1}^{5}a_{1i}x_i,\cdots,\sum_{i=1}^{5}a_{5i}x_i)\, .$$
   
   Thus $[A]\in \Aut(X)$ if and only if $A(F)=\lambda F$, for some $\lambda\in \C^*$.
   \end{notation}

\begin{lemma}\label{lem:mukailemma}
Let $A\in \GL(5,\C)$. Suppose $[A]\in \Aut(X)$. Then the automorphism $[A]$ of $X$ is Gorenstein if and only if $A(F)=\Det(A) F$.
\end{lemma}

\begin{proof}
See, for instance, \cite[Lemma~2.1]{Mu88}.
\end{proof}

 \begin{definition}
  (1)  We say a subgroup $\widetilde{G}< \GL(5,\mathbb{C})$ is a {\it lifting} of $G$ if $\widetilde{G}$ and $G$ are isomorphic via the natural projection $\pi: \GL(5,\C)\rightarrow \PGL(5,\C).$ We call $G$ {\it liftable} if $G$ admits a lifting.
 
 (2) We say $G$ is $F$-{\it liftable} if the following two conditions are satisfied:
  
  1) $G$ admits a lifting $\widetilde{G}< \GL(5,\mathbb{C})$; and
  
  2)  $A(F)=F$, for all $A$ in $\widetilde{G}$.

  In this case, we say $\widetilde{G}$ is an $F$-{\it lifting} of $G$.
  
 (3)  Let $h$ be an element in $G$. As a special case, we say $H\in \GL(5,\mathbb{C})$ is an $F$-{\it lifting} of $h$ if $\pi (H)=h$ and the group $\langle H\rangle $ is an $F$-lifting of the group $\langle h\rangle .$ 
  \end{definition}
  
 \begin{example}
   Let $F=x_1^5+x_2^5+x_3^5+x_4^5+x_5^5$  and let $G$ be the subgroup of $\PGL(5,\mathbb{C})$ generated by $[A_1]$ and $[A_2]$, where $$A_1=
    \begin{pmatrix} 
    0&1&0&0&0 \\ 
    0&0&1&0&0 \\ 
    0&0&0&1&0 \\ 
    0&0&0&0&1 \\ 
    1&0&0&0&0 \\ 
    \end{pmatrix},\,\, A_2=
    \begin{pmatrix} 
    1&0&0&0&0 \\ 
    0&\xi_5&0&0&0 \\ 
    0&0&\xi_5^2&0&0 \\ 
    0&0&0&\xi_5^3&0 \\ 
    0&0&0&0&\xi_5^4 \\ 
    \end{pmatrix}.$$ Then  $G$ is clearly a Gorenstein subgroup of $\Aut(X)$ and $G$ is not $F$-liftable (in fact, $G$ is not liftable).

\end{example}
    
\begin{theorem}\label{thm:lift}
(\cite[Theorems 4.8, 4.9, and 8.3]{OY15}) The group $G$ is $F$-liftable if and only if a Sylow $5$-subgroup $G_5$ is $F$-liftable. Moreover, $G$ is $F$-liftable if all the following three conditions are satisfied:

(1) $G$ is a finite solvable groups with  $|G|\leq 2000$,

(2) $|G|$ divides $2^63^25^2$, and

(3) $|G_2|>1$.
\end{theorem}

\begin{lemma}\label{lem:c2}
Let $g\in G$. Suppose $g$ is of order $2$. Then $h^{1,\;1}_g(X,G)=2$ and $e(X^g)=-8$.
\end{lemma}

\begin{proof}
By Theorem \ref{thm:lift}, $g$ has an $F$-lifting, say $A$. Then $\Ord(A)=2$ and $A(F)=F$. By Lemma \ref{lem:mukailemma}, $\Det(A)=1$. By \cite[Lemma 5.3]{OY15}, we may assume $A=\Diag(-1,-1,1,1,1)$. By $A(F)=F$ we have $F\in (x_3,x_4,x_5)$ (i.e., $F$ is in the ideal of $\C[x_1,...,x_5]$  generated by $x_3,x_4,x_5$). By \cite[Proposition 3.4]{OY15} $F\notin (x_1,x_2)$. Therefore, the fixed locus $X^g=X_1(g)\cup X_2(g)$, where $X_1(g): x_1=x_2=F=0$ is a smooth plane curve of degree $5$ and genus $6$, and $X_2(g): x_3=x_4=x_5=F=0$ is a line in $\P^4$. Thus $e(X^g)=e(X_1(g))+e(X_2(g))=-10+2=-8.$

For any point $x\in X_1(g)$, the eigenvalues of $g$ in the holomorphic tangent space $T_x$ are $$-1,-1,1$$ Thus $F_1(g)=1$. Similarly, $F_2(g)=1$. Then, by definition,  $X_0^g=X_2^g=\emptyset$, $X_1^g=X_1(g)\cup X_2(g)$. For any $h\in C(g)$, $h(X_i(g))=X_i(g),i=1,2$ since $X_1(g)$ and $X_2(g)$ are not isomorphic to each other. So $H_{C(g)}^{0,\;0}(X_1^g)=H^0(X_1^g)$. Therefore, $h_g^{1,\;1}(X,\; G)=h_{C(g)}^{0,\;0}(X_1^g)=2$.

\end{proof}

\begin{lemma}\label{lem:c3}
Let $g\in G$. Suppose $g$ is of order $3$. Then $g$ has an $F$-lifting, say $A$. Moreover, exactly one of the following two cases happens:

(1) The trace of $A$ is $2$ (i.e., $\Tr(A)$=2), $e(X^g)=-8$, $h^{1,\;1}_g(X,\;G)=2$, or

(2) $\Tr(A)=-1$, $e(X^g)=4$, $h^{1,\;1}_g(X,\;G)=2$.

\end{lemma}

\begin{proof}
By Theorem \ref{thm:lift}, an $F$-lifting $A$ exists. Then $\Ord(A)=3$ and $A(F)=F$. By Lemma \ref{lem:mukailemma}, $\Det(A)=1$. By \cite[Lemma 5.4]{OY15}, we may assume $A=\Diag(\xi_3,\xi_3^2,1,1,1)$ or $\Diag(\xi_3,\xi_3^2,\xi_3,\xi_3^2,1)$. 

Case (1) $A=\Diag(\xi_3,\xi_3^2,1,1,1)$. Then obviously $\Tr(A)=2$. By $A(F)=F$ we have $x_1^5,x_2^5\notin F$. Thus the fixed locus $X^g=X_1(g)\cup X_2(g)\cup X_3(g)$, where $X_1(g): x_1=x_2=F=0$ a smooth plane curve of degree $5$ and genus $6$, $X_2(g)=\{(0:1:0:0:0)\}$ a point, and $X_3(g)=\{(1:0:0:0:0)\}$ also a point. Thus $e(X^g)=e(X_1(g))+e(X_2(g))+e(X_3(g))=-10+1+1=-8.$  For any point $x\in X_1(g)$, the eigenvalues of $g$ in the holomorphic tangent space $T_x$ are $$\xi_3,\;\xi_3^2,\;1.$$ Thus $F_1(g)=1$. For $x=(0:1:0:0:0)\in X_2(g)$, the eigenvalues of $g$ in the holomorphic tangent space $T_x$ are $$\xi_3,\;\xi_3,\;\xi_3.$$ Thus $F_2(g)=1$.  For $x=(1:0:0:0:0)\in X_3(g)$, the eigenvalues of $g$ in the holomorphic tangent space $T_x$ are $$\xi_3^2,\;\xi_3^2,\;\xi_3^2.$$ Thus $F_3(g)=2$ Then, by definition,  $X_0^g=\emptyset$,  $X_1^g=X_1(g)\cup X_2(g)$, and  $X_2^g=X_3(g)$. For any $h\in C(g)$, $h(X_1(g))=X_1(g),$ since $X_1(g)$ is  isomorphic to neither  $X_2(g)$ nor $X_3(g)$. Moreover, $h(X_i(g))=X_i(g)$, $i=2,3$ since $F_2(g)$ and $F_3(g)$ are different. So $H_{C(g)}^{0,\;0}(X_1^g)=H^0(X_1^g)$. Therefore, $h_g^{1,\;1}(X,\; G)=h_{C(g)}^{0,\;0}(X_1^g)=2$.

Case (2)  $A=\Diag(\xi_3,\xi_3^2,\xi_3,\xi_3^2,1)$. Then $\Tr(A)=-1$. By $A(F)=F$ and \cite[Proposition 3.3]{OY15},  $x_5^5\in F$ (i.e., the coefficient of the monomial $x_5^5$ in $F$ is not zero). Thus the point $(0:0:0:0:1)\notin X$. So the fixed locus $X^g=X_1(g)\cup X_2(g)$, where $X_1(g): x_1=x_3=x_5=F=0$ is a line in $\P^4$, and $X_2(g): x_2=x_4=x_5=F=0$ is also a line in $\P^4$. Thus $e(X^g)=e(X_1(g))+e(X_2(g))=2+2=4.$  For any point $x\in X_1(g)$, the eigenvalues of $g$ in the holomorphic tangent space $T_x$ are $$\xi_3,\;\xi_3^2,\;1$$ Thus $F_1(g)=1$. Similarly, $F_2(g)=1$. For any $h\in C(g)$, $h(X_i(g))=X_i(g)$, $i=1,2$ (by considering the eigenvalues of $g$ in the holomorphic tangent space of $\P^4$ at points in $X_i(g)$).  Therefore, $h_g^{1,\;1}(X,\; G)=h_{C(g)}^{0,\;0}(X_1^g)=2$.

\end{proof}

\begin{definition}
Let $A\in \GL(5,\C)$. We define $m(A):={\rm max}\{\text{ multiplicities of eigenvalues of } A \}$. For example, if $A=\Diag(1,1,\xi_5,\xi_5,\xi_5^3)$, then $m(A)=2$.
\end{definition}

\begin{lemma}\label{lem:c5}
Let $g\in G$. Suppose $g$ is of order $5$. Then $g$ has an $F$-lifting, say $A$. Moreover, exactly one of the following three cases happens:

(1) $m(A)=1,\Tr(A)=0$, $e(X^g)=0$, $h^{1,\;1}_g(X,\;G)=0$,

(2) $m(A)=2$, $e(X^g)=10$, or

(3) $m(A)=3$, $e(X^g)=-10$, $h^{1,\;1}_g(X,\;G)=1$. 

\begin{remark}
In case (2), $h^{1,\;1}_g(X,\;G)$ depends on $C(g)$.

\end{remark}
\end{lemma}

\begin{proof}
By \cite[Lemma 4.13]{OY15} and Lemma \ref{lem:mukailemma}, an $F$-lifting, say $A$, of $g$ exists. Then $\Ord(A)=5$, $A(F)=F$ and $\Det(A)=1$. By an easy computation there are exactly three possibilities: (1) $m(A)=1$, (2) $m(A)=2$, or (3) $m(A)=3$.

Case (1) $m(A)=1$. Then $\Tr(A)=0$ and we may assume $A=\Diag(1,\xi_5,\xi_5^2,\xi_5^3,\xi_5^4)$. Thus $X^g=\emptyset$ and $e(X^g)=h^{1,\;1}_g(X,\;G)=0$.

Case (2)  $m(A)=2$. Then we may assume $A=\Diag(1,1,\xi_5,\xi_5,\xi_5^3)$ or $A=\Diag(1,1,\xi_5^2,\xi_5^2,\xi_5)$. Thus $X^g$ is a set of ten points and $e(X^g)=10$.

Case (3)  $m(A)=3$. Then we may assume $A=\Diag(1,1,1,\xi_5,\xi_5^4)$ or $A=\Diag(1,1,1,\xi_5^2,\xi_5^3)$. Thus $X^g$: $x_4=x_5=F=0$ a smooth plane curve of degree $5$ and genus $6$. Therefore, $e(X^g)=-10$ and $h^{1,\;1}_g(X,\;G)=1$.

\end{proof}

\begin{lemma}\label{lem:c6}
Let $g\in G$. Suppose $g$ is of order $6$. Then $h^{1,\;1}_g(X,G)=2$ and $e(X^g)=4$.
\end{lemma}

\begin{proof}
By Theorem \ref{thm:lift}, $g$ has an $F$-lifting, say $A$. Then $\Ord(A)=6$ and $A(F)=F$. By Lemma \ref{lem:mukailemma}, $\Det(A)=1$. By \cite[Lemma 8.6]{OY15}, we may assume $A=\Diag(\xi_3,\xi_3^2,-\xi_3,-\xi_3^2,1)$. Then $X^g=\{(1:0:0:0:0),(0:1:0:0:0),(0:0:1:0:0),(0:0:0:1:0)\}$. Thus $e(X^g)=4$ and $h^{1,\;1}_g(X,G)=2$ (cf. the proof of Lemma \ref{lem:c3}).
\end{proof}

\begin{lemma}\label{lem:c10}
Let $g\in G$. Suppose $g$ is of order $10$. Then $h^{1,\;1}_g(X,G)=1$ and $e(X^g)=2$.
\end{lemma}

\begin{proof}
By Theorem \ref{thm:lift}, $g$ has an $F$-lifting, say $A$. Then $\Ord(A)=10$ and $A(F)=F$. By Lemma \ref{lem:mukailemma}, $\Det(A)=1$. By \cite[Lemma 8.10]{OY15}, we may assume $A=\Diag(-1,1,-\xi_5^i,\xi_5^i,\xi_5^{3i})$, $i=1,2,3$ or $4$. Then $X^g=\{(1:0:0:0:0),(0:0:1:0:0)\}$. Thus $e(X^g)=2$ and $h^{1,\;1}_g(X,G)=1$ (cf. the proof of Lemma \ref{lem:c3}).
\end{proof}

\begin{lemma}\label{lem:c10}
Let $g\in G$. Suppose $g$ is of order $15$. Then $g$ has an $F$-lifting, say $A$. Moreover, either one of the following three cases is true:

(1) $\Tr(A^5)=-1$,  $m(A^3)=2$, $e(X^g)=4$, $h^{1,\;1}_g(X,\;G)=2$, or

(2)  $\Tr(A^5)=2$, $m(A^3)=2$, $e(X^g)=7$, or

(3) $\Tr(A^5)=2$, $m(A^3)=3$, $e(X^g)=2$,  $h^{1,\;1}_g(X,\;G)=1$. 

\begin{remark}
In case (2), $h^{1,\;1}_g(X,\;G)$ depends on $C(g)$.

\end{remark}
\end{lemma}

\begin{proof}
By Theorem \ref{thm:lift} and Lemma \ref{lem:c5}, an $F$-lifting of $g$ exists. Then by the proofs of Lemmas \ref{lem:c3} and \ref{lem:c5}, we may assume $A$ is either (1) $\Diag(\xi_3,\xi_3^2,\xi_3\xi_5^i,\xi_3^2\xi_5^i,\xi_5^{3i})$, $1\leq i\leq 4$, or (2)   $\Diag(\xi_3,\xi_3^2,\xi_5^i,\xi_5^i,\xi_5^{3i})$, $1\leq i\leq 4$, or (3)  $\Diag(\xi_3,\xi_3^2,\xi_5^i,\xi_5^{4i},1)$, $1\leq i\leq 2$.

Case (1) $\Diag(\xi_3,\xi_3^2,\xi_3\xi_5^i,\xi_3^2\xi_5^i,\xi_5^{3i})$, $1\leq i\leq 4$. Clearly $\Tr(A^5)=-1,m(A^3)=2$. Then $X^g=\{(1:0:0:0:0),(0:1:0:0:0),(0:0:1:0:0),(0:0:0:1:0)\}$. Thus $e(X^g)=4$ and  $h^{1,\;1}_g(X,\;G)=2$.

Case (2) $\Diag(\xi_3,\xi_3^2,\xi_5^i,\xi_5^i,\xi_5^{3i})$, $1\leq i\leq 4$. Clearly $\Tr(A^5)=2,m(A^3)=2$. Then $X^g=\{(1:0:0:0:0),(0:1:0:0:0)\}\cup W$, where $W:x_1=x_2=x_5=F=0$ a set of five points. Thus $e(X^g)=7$.

Case (3)  $\Diag(\xi_3,\xi_3^2,\xi_5^i,\xi_5^{4i},1)$, $1\leq i\leq 2$. Clearly $\Tr(A^5)=2,m(A^3)=3$. Then $X^g=\{(1:0:0:0:0),(0:1:0:0:0)\}$. Thus $e(X^g)=2$ and  $h^{1,\;1}_g(X,\;G)=1$.
\end{proof}

For any $g\in G$ with $\Ord(g)$ different from above (e.g. $\Ord(g)=4,12,13$, and so on),  we can determine $X^g$, $e(X^g)$ and $h_g^{1,\;1}(X,\; G)$ in the same way (the details are left to the readers). More precisely, all such results are summarized in Table \ref{tab:g}.

\begin{table}[htdp]
\caption{ Fixed loci of Gorenstein automorphisms of smooth quintic threefolds}\label{tab:g}
\begin{center}
{\footnotesize

\begin{tabular}{|c| p{5cm}|c|c|c|}

\hline 
$\Ord(g)$  & properties of $A$ & $X^g$&$e(X^g)$&$h_g^{1,\;1}(X,\; G)$\\
\hline 
 1 & $\Tr(A)=5$  & $X$&-200&  1\\
\hline 
 2 &$\Tr(A)=1$   & a line and $C$ &-8 &2\\
\hline
 3 &$\Tr(A)=2$   & $C$ and two points & -8&2\\
\hline
 3 &$\Tr(A)=-1$   & two lines & 4&2\\
\hline 
 4&$\Tr(A)=3$   & two points and $C$ & -8&2\\
 
\hline 
 4&$\Tr(A)=1+2\xi_4$   & six points and a line & 8&depending on $C(g)$\\
\hline 
 4 &$\Tr(A)=1-2\xi_4$   & six points and a line & 8&2\\
\hline 
 5 &$m(A)=1$   & $\emptyset$ & 0&0 \\
\hline 
 5 &$m(A)=2$   & ten points & 10& depending on $C(g)$ \\
\hline 
5  &  $m(A)=3$  &C&-10&1  \\
\hline 
6 &$\Tr(A)=1$   & four points & 4&2 \\
\hline 
 8 &$\Tr(A)=1+\xi_8+\xi_8^3$ & eight points  & 8&5\\
\hline 
 8 &$\Tr(A)=1+\xi_8^5+\xi_8^7$   & eight points & 8&3\\
\hline 
 8 &$\Tr(A)=\pm \xi_4$   & four points & 4&2\\
\hline 
10 & $(\Tr(A))^5=1$    &two points &2&1 \\
\hline 
12 & $m(A)=1$    &four points &4&2 \\
\hline 
13 & $\Tr(A)=\xi_{13}+\xi_{13}^9+\xi_{13}^3+2$ or $\xi_{13}^2+\xi_{13}^5+\xi_{13}^6+2$    &eight points &8&depending on $C(g)$ \\
\hline 
13 & $\Tr(A)=\xi_{13}^4+\xi_{13}^{10}+\xi_{13}^{12}+2$    &eight points &8&2 \\
\hline 
13 & $\Tr(A)=\xi_{13}^7+\xi_{13}^{11}+\xi_{13}^8+2$    &eight points &8&3 \\
\hline 
15 &$\Tr(A^5)=-1,m(A^3)=2$   & four points & 4&2 \\
\hline 
15 &$\Tr(A^5)=2,m(A^3)=2$   &seven points & 7&depending on $C(g)$ \\
\hline 
15 &$\Tr(A^5)=2,m(A^3)=3$   & two points & 2&1 \\
\hline 
17& $m(A)=1$    &four points &4&2 \\
\hline 
20& $m(A)=1$    &two points &2&1 \\
\hline 
39 & $\Tr(A)=\xi_{13}+\xi_{13}^{9}+\xi_{13}^3-1$ or $\xi_{13}^7+\xi_{13}^{11}+\xi_{13}^8-1$   &five points &5&3 \\
\hline 
39 & $\Tr(A)=\xi_{13}^2+\xi_{13}^{5}+\xi_{13}^6-1$ or $\xi_{13}^4+\xi_{13}^{10}+\xi_{13}^{12}-1$   &five points &5&2 \\
\hline 
41 & $\Tr(A)=\xi_{41}+\xi_{41}^{37}+\xi_{41}^{16}+\xi_{41}^{18}+\xi_{41}^{10}$ or $\xi_{41}^2+\xi_{41}^{33}+\xi_{41}^{32}+\xi_{41}^{36}+\xi_{41}^{20}$ or $\xi_{41}^6+\xi_{41}^{17}+\xi_{41}^{14}+\xi_{41}^{26}+\xi_{41}^{19}$ or $\xi_{41}^{11}+\xi_{41}^{38}+\xi_{41}^{12}+\xi_{41}^{34}+\xi_{41}^{28}$  &five points &5&3 \\
\hline 
41 & $\Tr(A)=\xi_{41}^3+\xi_{41}^{29}+\xi_{41}^{7}+\xi_{41}^{13}+\xi_{41}^{30}$ or $\xi_{41}^4+\xi_{41}^{25}+\xi_{41}^{23}+\xi_{41}^{31}+\xi_{41}^{40}$ or $\xi_{41}^5+\xi_{41}^{21}+\xi_{41}^{39}+\xi_{41}^{8}+\xi_{41}^{9}$ or $\xi_{41}^{15}+\xi_{41}^{22}+\xi_{41}^{35}+\xi_{41}^{24}+\xi_{41}^{27}$  &five points &5&2 \\
\hline 
51 & $\Tr(A)=\xi_{51}^{23}+\xi_{51}^{10}+\xi_{51}^{11}+\xi_{51}^7+1$ or $\xi_{51}^{26}+\xi_{51}^{49}+\xi_{51}^{8}+\xi_{51}^{19}+1$  &four points &4&1 \\
\hline 
51 & $\Tr(A)=\xi_{51}^{20}+\xi_{51}^{22}+\xi_{51}^{14}+\xi_{51}^{46}+1$ or $\xi_{51}^{37}+\xi_{51}^{5}+\xi_{51}^{31}+\xi_{51}^{29}+1$ or $\xi_{51}+\xi_{51}^{47}+\xi_{51}^{16}+\xi_{51}^{38}+1$ or $\xi_{51}^{35}+\xi_{51}^{13}+\xi_{51}^{50}+\xi_{51}^{4}+1$ &four points &4&2 \\
\hline 
51 & $\Tr(A)=\xi_{51}^{40}+\xi_{51}^{44}+\xi_{51}^{28}+\xi_{51}^{41}+1$ or $\xi_{51}^{43}+\xi_{51}^{32}+\xi_{51}^{25}+\xi_{51}^{2}+1$  &four points &4&3 \\
\hline 

65 & $\Tr(A)=\xi_{13}+\xi_{13}^{9}+\xi_{13}^{3}+\xi_{5}+\xi_{5}^{4}$ or $\xi_{13}^7+\xi_{13}^{11}+\xi_{13}^{8}+\xi_{5}+\xi_{5}^{4}$ or $\xi_{13}+\xi_{13}^{9}+\xi_{13}^{3}+\xi_{5}^{2}+\xi_{5}^{3}$ or $\xi_{13}^{2}+\xi_{13}^{5}+\xi_{13}^{6}+\xi_{5}^{2}+\xi_{5}^{3}$  &three points &3&2 \\
\hline 
65 & $\Tr(A)=\xi_{13}^2+\xi_{13}^{5}+\xi_{13}^{6}+\xi_{5}+\xi_{5}^{4}$ or $\xi_{13}^4+\xi_{13}^{10}+\xi_{13}^{12}+\xi_{5}^2+\xi_{5}^{3}$ or $\xi_{13}^7+\xi_{13}^{11}+\xi_{13}^{8}+\xi_{5}^{2}+\xi_{5}^{3}$ or $\xi_{13}^{4}+\xi_{13}^{10}+\xi_{13}^{12}+\xi_{5}+\xi_{5}^{4}$  &three points &3&1 \\
\hline 

\end{tabular}
}

\end{center}
\label{K3inCICY}

\vspace{3mm}
\begin{remark}
In this table, $g$ is an element in $G$, $A$ is an $F$-lifting of $g$, and we denote by $C$ some smooth plane curve of degree 5.
\end{remark}

\end{table}

For any subgroup $H<G$, we set $X^H:=\{x\in X\;|\; g(x)=x, {\rm \;for \;all\;} g\in H \}$.

\begin{lemma}\label{lem:c2c2}
Suppose $H<G$ and $H\cong C_2\ti C_2$. Then $e(X^H)=8$.
\end{lemma}

\begin{proof}
By Theorem \ref{thm:lift} $H$ has an $F$-lifting, say $\wt{H}$. By the proof of Lemma \ref{lem:c2}, we may assume $\wt{H}=\la A_1,A_2\ra$, where $A_1=\Diag(-1,-1,1,1,1)$, and $A_2=\Diag(1,-1,-1,1,1)$. Thus $X^H=\{(1:0:0:0:0),(0:1:0:0:0),(0:0:1:0:0)\}\cup W$, where $W:x_1=x_2=x_3=F=0$ a set of five points. Therefore, $e(X^H)=8$.
\end{proof}

\begin{lemma}\label{lem:c3c3}
Suppose $H<G$ and $H\cong C_3\ti C_3$. Then $e(X^H)=4$.
\end{lemma}

\begin{proof}
By Theorem \ref{thm:lift} $H$ has an $F$-lifting, say $\wt{H}$. By the proof of \cite[Lemma 5.4]{OY15}, we may assume $\wt{H}=\la A_1,A_2\ra$, where $A_1=\Diag(\xi_3,\xi_3^2,1,1,1)$, and $A_2=\Diag(1,1,\xi_3,\xi_3^2,1)$. Thus $X^H=\{(1:0:0:0:0),(0:1:0:0:0),(0:0:1:0:0),(0:0:0:1:0)\}$ and  $e(X^H)=4$.
\end{proof}

\begin{lemma}\label{lem:c5c5}
Suppose $H<G$ and $H\cong C_5\ti C_5$. Then one of the following three cases is true:

(1) $H$ is not $F$-liftable, $X^H=\emptyset$, $e(X^H)=0$, or

(2)  $H$ is  $F$-liftable, $X^H=\emptyset$, $e(X^H)=0$, or

(3) $H$ is  $F$-liftable, $X^H=$ a set of five points, $e(X^H)=5$.
\end{lemma}

\begin{proof}

If $H$ is not $F$-liftable, then by \cite[Lemma 4.14]{OY15} and Lemma \ref{lem:c5} we have $X^H=\emptyset$.  If $H$ is $F$-liftable, then by an easy computation $X^H$ is either an empty set or a set of five points.
\end{proof}

If $H<G$ and $H\cong C_4\ti C_2$, $C_{15}\ti C_3$, or $C_{15}\ti C_5$, we can determine the fixed locus $X^H$ in the same way as above. We summarize the results in Table \ref{tab:pairs}.

\begin{table}[htdp]
\caption{Fixed loci of pairs of commutating Gorenstein automorphisms of quintics}\label{tab:pairs}

\begin{center}
{\footnotesize

\begin{tabular}{|c|c|c|}

\hline 
the group generated by $g$ and $h$ &$X^g\cap X^h$  & $e(X^g\cap X^h)$\\
\hline 
 $C_2\ti C_2$ & eight points  & 8 \\
\hline 
$C_4\ti C_2$ & eight points  & 8 \\
\hline 
  $C_3\ti C_3$ & four points  & 4 \\
\hline

 $C_5\ti C_5$ &$\emptyset$ & 0 \\
\hline 

 $C_5\ti C_5$ & five points  & 5 \\
\hline 

$C_{15}\ti C_3$ & four points  & 4 \\
\hline 
$C_{15}\ti C_5$ & two points  & 2 \\
\hline 
\end{tabular}
}

\end{center}
\label{K3inCICY}

\vspace{3mm}
\begin{remark}
In this table, $g,h\in G$ and $gh=hg$. Note that if the group generated by $g$ and $h$ is a cyclic group, then both $X^g\cap X^h$ and $e(X^g\cap X^h)$ can be determined by Table \ref{tab:g}.
\end{remark}

\end{table}

\begin{theorem}\label{thm:main}
Let $a$ and $b$ be integers. If the pair of integers $(a,b)$ is in the following list, then there exists a smooth quintic threefold $X$ and a Gorenstein subgroup $G$ of $\Aut(X)$ such that  the pair of Hodge numbers $(h^{1,1}(\widehat{X/G}), h^{2,1}(\widehat{X/G}))$ of a crepant resolution $\widehat{X/G}$ of the quotient $X/G$ is the same as $(a, b)$: 

$(1,101)$, $(101,1)$, $(1,21)$, $(21,1)$, $(1,5)$, $(5,1)$, $(3,59)$, $(59,3)$, $(3,19)$, $(19,3)$, $(5,49)$, $(49,5)$, $(5,33)$, $(33,5)$, $(5,25)$, $(25,5)$, $(11,27)$, $(27,11)$, $(11,19)$, $(19,11)$, $(11,15)$, $(15,11)$, $(13,17)$, $(17,13)$, $(17,21)$, $(21,17)$, $(13,13)$, $(5,13)$, $(7,47)$, $(7,35)$, $(7,27)$, $(9,29)$, $(9,41)$, $(11,3)$, $(15,3)$, $(17,9)$, $(19,21)$, $(19,7)$, $(21,5)$, $(23,19)$, $(23,15)$, $(29,5)$, $(41,1)$.
\end{theorem}

\begin{remark}\label{rm:main}
Based on the main results of \cite{OY15}, it can be shown that the  list in Theorem \ref{thm:main} is in fact complete (i.e., there are no other pairs of integers except the above $43$ pairs appearing as the pair of Hodge numbers $(h^{1,1}(\widehat{X/G}), h^{2,1}(\widehat{X/G}))$).

Some of the pairs in the list in Theorem \ref{thm:main} represent ``new'' Calabi-Yau threefolds (here we refer to the database \cite{Ju} of the known Calabi-Yau threefolds). See Table \ref{tab:main} for more details.
\end{remark}

\begin{proof}
The proof is based on Table \ref{tab:main}. We first show two cases (i.e., $(a,b)=(15,11)$,  $(5,1)$) in details.

Case $(a,b)=(15,11)$.  As in Table \ref{tab:main}, let $X: F=x_1^4x_2+x_2^5+x_3^4x_4+x_4^5+x_5^5=0$  and let $G$ be the subgroup of $\PGL(5,\mathbb{C})$ generated by the following three matrices: $A_1=\Diag(-1,1,-1,1,1)$, $A_2=\Diag(\xi_5,\xi_5,\xi_5^4,\xi_5^4,1)$, $A_3=
    \begin{pmatrix} 
    0&0&1&0&0 \\ 
    0&0&0&1&0 \\ 
    1&0&0&0&0 \\ 
    0&1&0&0&0 \\ 
    0&0&0&0&1 \\ 
    \end{pmatrix}$. Then $G\cong D_{20}$ and $G$ acts on $X$ Gorensteinly. Let $\wt{G}:=\la A_1,A_2,A_3\ra <\GL(5,\C)$. Then $\wt{G}$ is an $F$-lifting of $G$. The group $G$ has exactly eight conjugacy classes and by Table \ref{tab:g} we obtain Table \ref{tab:D20}.
    
    \begin{table}[htdp]
\caption{Conjugacy classes of Gorenstein subgroup $G\cong D_{20}$ of $\Aut(X)$}\label{tab:D20}
\begin{center}
{\footnotesize

\begin{tabular}{|c| c|c|c|c|c|c|c|c|}

\hline 
$\{g\}$  & $\{[I_5]\}$ & $\{[A_1A_3]\}$&$\{[A_1]\}$& $\{[A_3]\}$&$\{[A_2]\}$&$\{[A_2^2]\}$&$\{[A_1A_2]\}$&$\{[A_1A_2^2]\}$ \\
\hline 
 \Ord(g) & 1  &2&2&  2&5&5&10&10\\
\hline 
 $h^{1,1}_g(X,G)$ &1  &2&2&  2&$\alpha_1$&$\alpha_2$&1&1\\
\hline
 $e(X^g)$ &-200  &-8& -8&-8&10&10&2&2\\
\hline
 $C(g)$ &$D_{20}$   & $C_2\ti C_2$ & $D_{20}$&$C_2\ti C_2$&$C_{10}$&$C_{10}$&$C_{10}$&$C_{10}$\\
\hline 
 \#(\{g\}) & 1  &5&1&  5&2&2&2&2\\
\hline

\end{tabular}
}
\end{center}
%\vspace{3mm}
%\begin{remark}
%In this table, $g$ is an element in $G$, $A$ is an $F$-lifting of $g$, and we denote by $C$ some smooth plane curve of degree 5.
%\end{remark}
\end{table}

By formula (\ref{eq:e(X,G)2}), Tables \ref{tab:D20} and  \ref{tab:pairs}, $$e(X,G)=\frac{1}{20}( \sum_{h\in G}e(X^h)+5\sum_{h\in C([A_1A_3])}e(X^{[A_1A_3]}\cap X^h)
+$$ $$ \sum_{h\in C([A_1])}e(X^{[A_1]}\cap X^h)+5\sum_{h\in C([A_3])}e(X^{[A_3]}\cap X^h)+2 \sum_{h\in C([A_2])}e(X^{[A_2]}\cap X^h)+$$ $$2 \sum_{h\in C([A_2^2])}e(X^{[A_2^2]}\cap X^h)+2 \sum_{h\in C([A_1A_2])}e(X^{[A_1A_2]}\cap X^h)+2 \sum_{h\in C([A_1A_2^2])}e(X^{[A_1A_2^2]}\cap X^h))$$ $$=\frac{1}{20}(-240+0+80+0+120+120+40+40)=8.$$ 

In order to compute $h^{1,1}(X,G)$, we need to compute $\alpha_1$ and $\alpha_2$. Let $g=[A_2]\in G$. The fixed locus $X^{g}$ consists of ten points. Moreover, by computing the fermion shift numbers at those points, we obtain $X^{g}_0=\emptyset$, $X^{g}_1=\{(0:0:1:0:0), (0:0:\xi_8:1:0),  (0:0:\xi_8^3:1:0),  (0:0:\xi_8^5:1:0),  (0:0:\xi_8^7:1:0)\}$,  and $X^{g}_2=\{(1:0:0:0:0), (\xi_8:1:0:0:0),  (\xi_8^3:1:0:0:0),  (\xi_8^5:1:0:0:0),  (\xi_8^7:1:0:0:0)\}$. Note that the centralizer group $C([A_2])=\la [A_1], [A_2]\ra$. Then  the set $X^g_1$ is naturally acted on by  $C([A_2])$ and the number of orbits of this group action is $3$. Thus, by formula (\ref{eq:hg}), $\alpha_1=3$. Similarly, $\alpha_2=3$. Therefore, by formula (\ref{eq:h(X,G)}), $$h^{1,1}(X,G)=1+2+2+2+3+3+1+1=15.$$

By Theorem \ref{thm:mckay}, $e(\widehat{X/G})=e(X,\; G)=8$, $h^{1,1}(\widehat{X/G})=h^{1,1}(X,G)=15$. Since $\widehat{X/G}$ is a Calabi-Yau threefold, it follows that $h^{2,1}(\widehat{X/G})=h^{1,1}(\widehat{X/G})-\frac{1}{2}e(\widehat{X/G})=15-4=11$. Therefore, the case $(a,b)=(15,11)$ is proved.

Case $(a,b)=(5,1)$. As in Table \ref{tab:main}, let $X: F=x_1^5+x_2^5+x_3^5+x_4^5+x_5^5=0$  and let $G$ be the subgroup of $\PGL(5,\mathbb{C})$ generated by the following two matrices: $A_1=\Diag(1,\xi_5,\xi_5^4,\xi_5^4,\xi_5)$, $A_2=
    \begin{pmatrix} 
    0&1&0&0&0 \\ 
    0&0&1&0&0 \\ 
    0&0&0&1&0 \\ 
    0&0&0&0&1 \\ 
    1&0&0&0&0 \\ 
    \end{pmatrix}$. Then $G\cong (C_5\ti C_5)\rt C_5$. Note that the order of an element of $G$ is either $1$ or $5$. Moreover,  $G$ has $29$ conjugacy classes: $\{[I_5]\}$, $\{[A_1]\}$, $\{[A_1^2]\}$, $\{[A_1^3]\}$, $\{[A_1^4]\}$, and $24$ other conjugacy classes which are represented by fixed point free automorphisms of $X$.  Then by Table \ref{tab:g} we obtain Table \ref{tab:125}.
    
      \begin{table}[htdp]
\caption{Conjugacy classes of Gorenstein subgroup $G\cong (C_5\ti C_5)\rt C_5$ of $\Aut(X)$}\label{tab:125}
\begin{center}
{\footnotesize

\begin{tabular}{|c| c|c|c|c|c|c|}

\hline 
$\{g\}$  & $\{[I_5]\}$ & $\{[A_1]\}$&$\{[A_1^2]\}$& $\{[A_1^3]\}$&$\{[A_1^4]\}$& 24 others\\
\hline 
 \Ord(g) & 1  &5&5&  5&5&5\\
\hline 
 $h^{1,1}_g(X,G)$ &1  &$\beta_1$&$\beta_2$&  $\beta_3$&$\beta_4$&0\\
\hline
 $e(X^g)$ &-200  &10& 10&10&10&0\\
\hline
 $C(g)$ &$(C_5\ti C_5)\rt C_5$   & $C_5\ti C_5$ & $C_5\ti C_5$&$C_5\ti C_5$&$C_5\ti C_5$& \\
\hline 
 \#(\{g\}) & 1  &5&5&  5&5& \\
\hline

\end{tabular}
}
\end{center}
%\vspace{3mm}
%\begin{remark}
%In this table, $g$ is an element in $G$, $A$ is an $F$-lifting of $g$, and we denote by $C$ some smooth plane curve of degree 5.
%\end{remark}
\end{table}

Let $H=\la [A_1], [A_3]\ra<G$, where $A_3=\Diag(1,\xi_5,\xi_5^2,\xi_5^3,\xi_5^4)$.  Note that $C([A_1])=H$ and $e(X^{[A_1]}\cap X^{[A_3]})=e(X^{[A_3]})=0$. Let $g=[A_1]\in G$. Thus, by formula (\ref{eq:e(X,G)2}) and Table \ref{tab:125} $$e(X,G)=\frac{1}{125}(\sum_{h\in G}e(X^h)+\sum_{i=1}^{4}(5\sum_{h\in H}e(X^{g^i}\cap X^h)))=\frac{1}{125}(0+4\cdot 5\cdot 50)=8.$$ Similar to the previous case, by an easy computation, $\beta_1=\beta_2=\beta_3=\beta_4=1$. Therefore, $$h^{1,1}(X,G)=1+1+1+1+1=5.$$ Then by Theorem \ref{thm:mckay} $e(\widehat{X/G})=8$, $h^{1,1}(\widehat{X/G})=5$, and $h^{2,1}(\widehat{X/G})=1$. Notice that $H$ is in fact the normal subgroup of $G$ generated by those elements which have fixed points. Then by Theorem \ref{thm:pi1} $$\pi_1(\widehat{X/G})\cong G/H\cong C_5.$$

So far we have proved only two cases of Theorem \ref{thm:main}. However, the other cases can be proved in the same way (i.e., firstly we use Table \ref{tab:main} to find suitable $X$ and $G$; secondly we analyze conjugacy classes of $G$ and use Tables \ref{tab:g} and \ref{tab:pairs} to obtain tables like Tables \ref{tab:D20} and \ref{tab:125}; lastly we use formulas (\ref{eq:h(X,G)} ), (\ref{eq:e(X,G)2}) and Theorem \ref{thm:mckay} to compute the numbers which we want to know). Thus, we may leave the detailed proofs for the readers. 
\end{proof}

\begin{remark}
It is possible to determine the Hodge numbers of $\widehat{X/G}$ by first computing the representation of $G$ on the  cohomology of $X$ (cf. \cite{St12}). However, our method of determining those numbers are based on Tables \ref{tab:g}, \ref{tab:pairs} and formulas (\ref{eq:h(X,G)}),  (\ref{eq:e(X,G)2}) and hence such computation is not required. 
\end{remark}

\appendix

\section{The table}

In Table \ref{tab:main}, we give many examples of a pair $(X,\;G)$, where $X$ is a smooth quintic threefold and $G$ is a Gorenstein subgroup of $\Aut(X)$. We compute the Hodge numbers $(h^{1,1},\;h^{2,1})$ ($:=(h^{1,1}(\widehat{X/G}), \;h^{2,1}(\widehat{X/G}))$) of a crepant resolution $\widehat{X/G}$ of the quotient $X/G$. In the last column  ``new'' means that the pair $(h^{1,1},\;h^{2,1})$ is not contained in the list \cite{Ju} of Hodge numbers of known Calabi-Yau threefolds.

We need to explain some notations in Table \ref{tab:main}. In the second column, for simplicity, we use  elements in the symmetric group $S_5$ to represent the corresponding permutation matrices. For examples, $(12)(34)$ means the matrix $\begin{pmatrix} 
    0&1&0&0&0 \\ 
    1&0&0&0&0 \\ 
    0&0&0&1&0 \\ 
    0&0&1&0&0 \\ 
    0&0&0&0&1 \\ 
    \end{pmatrix}$, $(123)$ means the matrix $\begin{pmatrix} 
    0&1&0&0&0 \\ 
    0&0&1&0&0 \\ 
    1&0&0&0&0 \\ 
    0&0&0&1&0 \\ 
    0&0&0&0&1 \\ 
    \end{pmatrix}$, and so on.  In the third column, for $1\leq i\leq 22$, we use $X_i$ to denote the smooth quintic threefold in the $i$-th example of \cite[Example 2.1]{OY15}, and the matrices $A_{i,j}$ in the second column correspond to the matrices ``$A_j$'' in the same example of \cite[Example 2.1]{OY15}. For example, $X_{17}\subset \P^4$ is defined by $F=((x_1^4+x_2^4)+(2+4\xi_3^2)x_1^2x_2^2)x_3+(-(x_1^4+x_2^4)+(2+4\xi_3^2)x_1^2x_2^2)x_4+x_3^4x_4+x_4^4x_3+x_5^5=0$, and $A_{17,2}= \begin{pmatrix} 
    \xi_8^3&0&0&0&0 \\ 
    0&\xi_8&0&0&0 \\ 
    0&0&0&1&0 \\ 
    0&0&1&0&0 \\ 
    0&0&0&0&1 \\ 
    \end{pmatrix}$. We define two more smooth quintic threefolds $X_{C_4\ti C_2}: x_1^4x_4+x_2^4x_5+x_3^4x_4+x_4^5+x_5^5+x_1x_2x_3^3=0$ and $X_{C_8}:x_1^4x_2+x_2^4x_3+x_3^5+x_4^4x_2+x_5^4x_3+x_1x_4x_5^3+x_2^3x_5^2+x_1x_3x_4^3+x_1^3x_3x_4=0$. We leave  the sixth column blank if and only if the fundamental group $\pi_1(\widehat{X/G})$ is trivial.

\begin{center}
{\small
\begin{longtable}{| p{1.8cm}| p{6.6cm}| p{1.1cm}| p{1cm}|p{1.5cm}| p{1.2cm}|p{1.5cm}|}
\caption{Hodge numbers of $\widehat{X/G}$} \label{tab:main} \\
\hline 
 $G$ &generators of $G$ (as a subgroup of \PGL(5,C)) &$X$ & $e(\widehat{X/G})$ &$(h^{1,1},\;h^{2,1})$ 
&$\pi_1(\widehat{X/G})$& remarks on $(h^{1,1},\;h^{2,1})$\\
\hline 
\endfirsthead
\multicolumn{4}{c}%
{\tablename\ \thetable\ -- \textit{Continued from previous page}} \vspace{3.6mm}\\
\hline
 $G$ &generators of $G$ (as a subgroup of \PGL(5,C)) &$X$  & $e(\widehat{X/G})$ &$(h^{1,1},\;h^{2,1})$ 
&$\pi_1(\widehat{X/G})$& remarks on $(h^{1,1},\;h^{2,1})$ \\
\hline 
\endhead
\multicolumn{4}{r}{\textit{Continued on next page}} \\
\endfoot
\hline
\endlastfoot

trivial &  trivial &any & $-200$ &(1,101)&  &\\
\hline
$C_2$&  $(12)(34)$ &$X_1$   &$-112$  & $(3,59)$ &  &\\
\hline
$C_3$&  $(123)$ &$X_1$   &$-88$  & $(5,49)$ &  &\\
\hline
$C_3$&  $\Diag(\xi_3,\xi_3^2,\xi_3,\xi_3^2,1)$ &$X_{14}$   &$-56$  & $(5,33)$ &  &\\
\hline
$C_4$&  $\Diag(\xi_4,1,\xi_4^3,1,1)$ &$X_{3}$   &$-80$  & $(7,47)$ &  &\\
\hline 
$C_4$&  $\Diag(\xi_4,\xi_4,-1,1,1)$ &$X_{C_4\ti C_2}$   &$-32$  & $(11,27)$ &  &\\
\hline 
$C_2\ti C_2$&  $(12)(34)$, $(13)(24)$ &$X_{1}$   &$-56$  & $(7,35)$ &  &\\
\hline 
$C_5$&  $\Diag(1,\xi_5,\xi_5^4,1,1)$ &$X_{1}$   &$-88$  & $(5,49)$ &  &\\
\hline 
$C_5$&  $\Diag(\xi_5,\xi_5,\xi_5^4,\xi_5^4,1)$ &$X_{1}$   &$8$  & $(21,17)$ &  &\\
\hline 
$C_5$&  $\Diag(1,\xi_5,\xi_5^2,\xi_5^3,\xi_5^4)$ &$X_{1}$   &$-40$  & $(1,21)$ & $C_5$ &\\
\hline 
$S_3$&  $\Diag(\xi_3,\xi_3^2,1,1,1)$, $(12)(34)$ &$X_{14}$   &$-56$  & $(5,33)$ &  &\\
\hline 
$S_3$&  $\Diag(\xi_3,\xi_3^2,\xi_3,\xi_3^2,1)$,$(12)(34)$ &$X_{14}$   &$-40$  & $(5,25)$ &  &\\
\hline
$C_6$&  $\Diag(\xi_3,\xi_3^2,\xi_3,\xi_3^2,1)$,$(13)(24)$ &$X_{14}$   &$-16$  & $(11,19)$ &  &\\
\hline
$C_8$&  $A_{17,2}$&$X_{17}$   &$8$  & $(23,19)$ &  &\\
\hline
$C_8$&  $\Diag(\xi_8,-1,1,\xi_8^5,\xi_4^3)$ &$X_{C_8}$   &$8$  & $(17,13)$ &  &\\
\hline
$C_4\ti C_2$&  $\Diag(\xi_4,\xi_4,-1,1,1)$, $\Diag(1,-1,-1,1,1)$ &$X_{C_4\ti C_2}$   &$-4$  & $(19,21)$ & &\\
\hline
$D_8$&  $(12)(34)$, $\Diag(1,1,\xi_4,\xi_4^3,1)$ &$X_{20}$   &$-40$  & $(9,29)$ &  &\\
\hline
$Q_8$&  $A_{18,1}$, $\Diag(\xi_4^3,\xi_4,1,1,1)$ &$X_{18}$   &$-64$  & $(9,41)$ &  &\\
\hline
$C_3\ti C_3$&  $\Diag(\xi_3,\xi_3^2,1,1,1)$, $\Diag(1,1,\xi_3,\xi_3^2,1)$ &$X_{14}$   &$-8$  & $(17,21)$ &  &\\
\hline 
$D_{10}$&  $(12)(34)$,$\Diag(\xi_5,\xi_5^4,\xi_5,\xi_5^4,1)$ &$X_{1}$   &$-8$  & $(13,17)$ &  &\\
\hline
$D_{10}$&  $(12)(34)$, $\Diag(\xi_5,\xi_5^4,1,1,1)$ &$X_{1}$   &$-56$  & $(5,33)$ &  &\\
\hline
$D_{10}$&  $(12)(34)$,$\Diag(\xi_5,\xi_5^4,\xi_5^2,\xi_5^3,1)$ &$X_{1}$   &$-32$  & $(3,19)$ &  &\\
\hline 
$C_{10}$& $\Diag(-1,1,-\xi_5,\xi_5,\xi_5^3)$ &$X_{14}$   &$16$  & $(19,11)$ &  &\\
\hline
$C_{12}$&  $\Diag(\xi_3,\xi_3^2,\xi_3,\xi_3^2,1)$, $\Diag(1,1,\xi_4,\xi_4^3,1)$ &$X_{19}$   &$16$  & $(23,15)$ &  &\\
\hline
$A_{4}$&  $(12)(34)$, $(13)(24)$, $(123)$ &$X_{1}$   &$-40$  & $(7,27)$ &  &\\
\hline
$D_{12}$&  $\Diag(\xi_3,\xi_3^2,\xi_3,\xi_3^2,1)$,$(13)(24)$,$(12)(34)$ &$X_{14}$   &$-8$  & $(11,15)$ &  &\\
\hline
$C_{13}$&  $\Diag(\xi_{13},\xi_{13}^9,\xi_{13}^3,1,1)$ &$X_{16}$   &$88$  & $(49,5)$ &  &\\
\hline
$C_{15}$&  $\Diag(\xi_3,\xi_3^2,1,1,1)$, $\Diag(1,1,\xi_5,\xi_5,\xi_5^3)$ &$X_{14}$   &$88$  & $(49,5)$ &  &\\
\hline
$C_{15}$&  $\Diag(\xi_3,\xi_3^2,1,1,1)$, $\Diag(1,1,\xi_5,1,\xi_5^4)$ &$X_{5}$   &$-8$  & $(17,21)$ &  &\\
\hline
$C_{15}$&  $\Diag(\xi_3,\xi_3^2,\xi_3,\xi_3^2,1)$,$\Diag(1,1,\xi_5,\xi_5,\xi_5^3)$ &$X_{14}$   &$56$  & $(33,5)$ &  &\\
\hline
$QD_{16}$&  $A_{17,1}$, $A_{17,2}$ &$X_{17}$   &$-8$  & $(17,21)$ &  &\\
\hline
$C_{17}$&  $\Diag(\xi_{17},\xi_{17}^{-4},\xi_{17}^{16},\xi_{17}^4,1)$ &$X_{13}$   &$56$  & $(33,5)$ &  &\\
\hline
$C_{3}\ti S_3$&  $\Diag(\xi_3,\xi_3^2,1,1,1)$,$\Diag(1,1,\xi_3,\xi_3^2,1)$, $(13)(24)$ &$X_{14}$   &$8$  & $(17,13)$ &  &\\
\hline
$(C_{3}\ti C_3)\rt C_2$&  $\Diag(\xi_3,\xi_3^2,1,1,1)$,$\Diag(1,1,\xi_3,\xi_3^2,1)$, $(12)(34)$ &$X_{14}$   &$-16$  & $(11,19)$ &  &\\
\hline
$C_{20}$&  $\Diag(\xi_4,1,\xi_4^3,1,1)$, $\Diag(1,1,\xi_5,\xi_5,\xi_5^3)$ &$X_{3}$   &$32$  & $(27,11)$ &  &\\
\hline
$D_{20}$&  $\Diag(-\xi_5,\xi_5,-\xi_5^4,\xi_5^4,1)$,$(13)(24)$ &$X_{3}$   &$8$  & $(15,11)$ &  & new \\
\hline
$\SL(2,3)$&  $A_{17,1}$, $A_{17,3}$&$X_{17}$   &$0$  & $(13,13)$ &  &\\
\hline
$D_{24}$&  $\Diag(\xi_3,\xi_3^2,\xi_3,\xi_3^2,1)$,$\Diag(1,1,\xi_4,\xi_4^3,1)$, $(12)(34)$ &$X_{20}$   &$8$  & $(17,13)$ &  &\\
\hline
 $C_{5}\ti C_5$&  $\Diag(\xi_5,\xi_5^4,1,1,1)$, $\Diag(1,1,\xi_5,\xi_5^4,1)$ &$X_{1}$   &$-8$  & $(17,21)$ &  &\\
 \hline
 $C_{5}\ti C_5$&  $\Diag(\xi_5,\xi_5^4,1,1,1)$, $\Diag(1,1,\xi_5,\xi_5,\xi_5^3)$ &$X_{1}$   &$88$  & $(49,5)$ &  &\\
 \hline
 $C_{5}\ti C_5$&  $\Diag(1,\xi_5,\xi_5,\xi_5^4,\xi_5^4)$, $\Diag(1,1,\xi_5,\xi_5,\xi_5^3)$ &$X_{1}$   &$40$  & $(21,1)$ &  &\\
 \hline
 $C_{5}\ti C_5$&  $\Diag(1,\xi_5,\xi_5^2,\xi_5^3,\xi_5^4)$,$(12345)$ &$X_{1}$   &$-8$  & $(1,5)$ & $C_5\ti C_5$ &\\
 \hline
 $C_{5}\ti S_3$&  $\Diag(\xi_3,\xi_3^2,\xi_5,\xi_5,\xi_5^3)$,$(12)(34)$ &$X_{14}$   &$56$  & $(33,5)$ &  &\\
 \hline
 $C_{5}\ti S_3$&  $\Diag(\xi_3,\xi_3^2,\xi_{15}^8,\xi_{15}^{13},\xi_5^3)$,  $(12)(34)$ &$X_{14}$   &$40$  & $(25,5)$ &  & new \\
 \hline
 $C_{3}\ti D_{10}$&  $\Diag(\xi_{15}^{11},\xi_{15}^1,\xi_{15}^{14},\xi_{15}^{4},1)$, $(13)(24)$ &$X_{14}$   &$40$  & $(25,5)$ &  & new\\
 \hline
 $D_{30}$&  $\Diag(\xi_3,\xi_3^2,\xi_5,\xi_5^4,1)$,  $(12)(34)$ &$X_{5}$   &$-16$  & $(11,19)$ &  &\\
 \hline
  $ D_{30}$&  $\Diag(\xi_{15}^{11},\xi_{15}^1,\xi_{15}^{14},\xi_{15}^{4},1)$, $(14)(23)$ &$X_{14}$   &$16$  & $(19,11)$ &  &\\
  \hline
   $ D_{34}$&  $\Diag(\xi_{17},\xi_{17}^{-4},\xi_{17}^{-1},\xi_{17}^{4},1)$, $(13)(24)$ &$X_{13}$   &$16$  & $(19,11)$ &  &\\
  \hline
   $ S_{3}\ti S_3$&  $\Diag(\xi_{3},\xi_{3}^2,1,1,1)$,$\Diag(1,1,\xi_{3},\xi_{3}^{2},1)$, $(14)(23)$,$(13)(24)$ &$X_{14}$   &$16$  & $(17,9)$ &  & new\\
  \hline
   $ C_{13}\rt C_3$&  $\Diag(\xi_{13},\xi_{13}^9,\xi_{13}^{3},1,1)$, $(123)$ &$X_{16}$   &$8$  & $(21,17)$ &  &\\
  \hline
   $ C_{13}\rt C_3$&  $\Diag(\xi_{13},\xi_{13}^9,\xi_{13}^{3},1,1)$, $A_{16,2}A_{16,4}$ &$X_{16}$   &$40$  & $(21,1)$ &  &\\
  \hline
    $ C_{39}$&  $\Diag(\xi_{13},\xi_{13}^9,\xi_{13}^{3},1,1)$, $\Diag(1,1,1,\xi_3,\xi_3^2)$ &$X_{16}$   &$200$  & $(101,1)$ &  &\\
  \hline
  $ D_{40}$&  $\Diag(\xi_{20}^9,\xi_{5},\xi_{20}^{11},\xi_{5}^{4},1)$, $(13)(24)$ &$X_{3}$   &$16$  & $(19,11)$ &  &\\
  \hline
   $ C_{41}$&  $\Diag(\xi_{41},\xi_{41}^{-4},\xi_{41}^{16},\xi_{41}^{18},\xi_{41}^{10})$ &$X_{15}$   &$200$  & $(101,1)$ &  &\\
  \hline
   $ C_{15}\ti C_3$&  $\Diag(\xi_{3},\xi_{3}^2,\xi_5,\xi_5,\xi_5^3)$, $\Diag(1,1,\xi_{3},\xi_{3}^2,1)$ &$X_{14}$   &$200$  & $(101,1)$ &  &\\
  \hline
   $ \GL(2,3)$& $A_{17,1}$, $A_{17,2}$, $A_{17,3}$ &$X_{17}$   &$24$  & $(19,7)$ &  &\\
  \hline
   $ C_5\ti D_{10}$&  $\Diag(1,1,\xi_{5},\xi_{5},\xi_{5}^{3})$,$\Diag(\xi_{5},\xi_{5}^4,\xi_{5},\xi_{5}^{4},1)$, $(12)(34)$ &$X_{1}$   &$8$  & $(17,13)$ &  &\\
  \hline
    $ C_5\ti D_{10}$&  $\Diag(1,1,\xi_{5},\xi_{5},\xi_{5}^{3})$, $\Diag(\xi_{5},\xi_{5}^4,1,1,1)$, $(12)(34)$ &$X_{1}$   &$56$  & $(33,5)$ &  &\\
  \hline
    $ C_5\ti D_{10}$&  $\Diag(1,1,\xi_{5},\xi_{5},\xi_{5}^{3})$,$\Diag(\xi_{5},\xi_{5}^4,\xi_{5}^2,\xi_{5}^{3},1)$, $(12)(34)$ &$X_{1}$   &$32$  & $(19,3)$ &  & in \cite{St12}\\
  \hline
    $ (C_5\ti C_{5})\rt C_2$&  $\Diag(1,1,\xi_{5},\xi_{5}^4,1)$, $\Diag(\xi_{5},\xi_{5}^4,1,1,1)$, $(12)(34)$ &$X_{1}$   &$-16$  & $(11,19)$ &  &\\
  \hline
    $ C_{51}$&  $\Diag(\xi_{51}^{20},\xi_{51}^{22},\xi_{51}^{14},\xi_{51}^{46},1)$ &$X_{13}$   &$200$  & $(101,1)$ &  &\\
  \hline
    $A_5$&  $(12345)$,  $(123)$ &$X_{1}$   &$-16$  & $(5,13)$ & &\\
  \hline
    $ S_3\ti D_{10}$&  $\Diag(\xi_{15}^8,\xi_{15}^{13},\xi_{15}^2,\xi_{15}^{7},1)$, $(12)(34)$, $(13)(24)$ &$X_{14}$   &$32$  & $(21,5)$ &  & new\\
  \hline
    $ C_{65}$&  $\Diag(1,1,1,\xi_{5},\xi_{5}^{4})$, $\Diag(\xi_{13},\xi_{13}^{-4},\xi_{13}^3,1,1)$ &$X_{9}$   &$200$  & $(101,1)$ &  &\\
  \hline
    $ (C_5\ti C_{5})\rt C_3$&  $\Diag(1,1,\xi_{5},1,\xi_{5}^{4})$, $(345)$ &$X_{1}$   &$8$  & $(21,17)$ &  &\\
  \hline
   $ (C_5\ti C_{5})\rt C_3$&  $\Diag(1,1,\xi_{5},1,\xi_{5}^{4})$, $(345)\Diag(\xi_3,\xi_3^2,1,1,1)$ &$X_{5}$   &$40$  & $(21,1)$ &  &\\
  \hline
   $ C_{15}\ti C_{5}$&  $\Diag(1,1,\xi_{5},1,\xi_{5}^{4})$, $\Diag(\xi_3,\xi_3^2,1,\xi_5,\xi_5^4)$ &$X_{5}$   &$200$  & $(101,1)$ &  &\\
  \hline
   $ C_{3}\ti D_{30}$&  $\Diag(\xi_{15}^{8},\xi_{15}^{13},\xi_{5}^4,\xi_{5}^{4},1)$,$\Diag(1,1,\xi_{3},\xi_{3}^2,1)$, $(14)(23)$&$X_{14}$   &$112$  & $(59,3)$ &  &\\
  \hline
   $ C_5\ti((C_3\ti C_{3})\rt C_2)$&  $\Diag(\xi_3,\xi_3^2,\xi_{5},\xi_5, \xi_{5}^{3})$,$\Diag(1,1,\xi_{3},\xi_{3}^2,1)$, $(12)(34)$ &$X_{14}$   &$112$  & $(59,3)$ &  &\\
  \hline
   $D_{10}\ti D_{10}$&  $\Diag(\xi_5,\xi_5,\xi_{5}^4,\xi_{5}^{4},1)$,$\Diag(\xi_5,\xi_5^4,\xi_{5},\xi_{5}^{4},1)$, $(12)(34)$, $(13)(24)$ &$X_{1}$   &$16$  & $(17,9)$ &  &new\\
  \hline
  $C_{3}\ti D_{34}$&  $\Diag(\xi_3,\xi_3^2,\xi_{3},\xi_{3}^{2},1)$, $\Diag(\xi_{17},\xi_{17}^{-4},\xi_{17}^{16},\xi_{17}^{-1},1)$, $(13)(24)$ &$X_{13}$   &$112$  & $(59,3)$ &  &\\
  \hline
   $C_{3}\ti (C_{13}\rt C_3)$&  $\Diag(\xi_{13},\xi_{13}^{-4},\xi_{13}^3,\xi_{3},\xi_3^2)$, $(123)$ &$X_{16}$   &$88$  & $(49,5)$ &  &\\
  \hline
   $(C_{5}\ti C_{5})\rt C_5$&  $\Diag(1, \xi_5,\xi_5^4,\xi_{5}^4,\xi_{5})$,   $(12345)$ &$X_{1}$   &$8$  & $(5,1)$ & $C_5$ &in \cite{ALR90} \\
  \hline
   $C_{5}\ti C_{5}\ti C_5$&  $\Diag(1,\xi_5,1,1,\xi_{5}^4)$, $\Diag(1,1,\xi_5,1,\xi_{5}^{4})$, $\Diag(1,1,1,\xi_5,\xi_{5}^{4})$ &$X_{1}$   &$200$  & $(101,1)$ &  &\\
  \hline
   $((C_{5}\ti C_{5})\rt C_3)\rt C_2$&  $\Diag(1,1, \xi_5,1,\xi_{5}^4)$, $\Diag(1,1,1,\xi_5,\xi_{5}^{4})$,  $(345)$, $(12)(34)$ &$X_{1}$   &$16$  & $(19,11)$ &  &\\
  \hline
   $((C_{5}\ti C_{5})\rt C_3)\rt C_2$&  $\Diag(1,1, \xi_5,1,\xi_{5}^4)$, $\Diag(1,1,1,\xi_5,\xi_{5}^{4})$,  $(345)\Diag(\xi_3,\xi_3^2,1,1,1)$, $(12)(34)$ &$X_{5}$   &$32$  & $(19,3)$ &  & in \cite{St12}\\
  \hline
   $C_5\ti D_{30}$&  $\Diag(1,1, \xi_5,\xi_5,\xi_{5}^3)$,$\Diag(\xi_3,\xi_3^2,\xi_5,\xi_{5}^{4},1)$, $(12)(34)$ &$X_{5}$   &$112$  & $(59,3)$ &  &\\
  \hline
   $(C_5\ti((C_3\ti C_3)\rt C_2))\rt C_2$&  $\Diag(1,1, \xi_3,\xi_{3}^2,1)$, $\Diag(\xi_{15}^8,\xi_{15}^{13},\xi_5^4,\xi_{5}^{4},1)$,  $(12)(34)$, $(13)(24)$ &$X_{14}$   &$80$  & $(41,1)$ &  &new\\
  \hline
   $C_5\ti (C_{13}\rt C_3)$&  $\Diag(\xi_{13},\xi_{13}^{-4}, \xi_{13}^3,\xi_5,\xi_{5}^4)$,  $(123)$ &$X_{9}$   &$88$  & $(49,5)$ &  &\\
  \hline
   $C_{41}\rt C_5$& $\Diag(\xi_{41},\xi_{41}^{-4},\xi_{41}^{16},\xi_{41}^{18},\xi_{41}^{10})$, $(12345)$ &$X_{15}$   &$40$  & $(21,1)$ &  &\\
  \hline
   $C_3\ti ((C_5\ti C_5)\rt C_3)$&  $\Diag(\xi_3,\xi_3^2, \xi_5,1,\xi_5^4)$, $\Diag(1,1,1,\xi_{5},\xi_5^4)$,  $(345)$ &$X_{5}$   &$88$  & $(49,5)$ &  &\\
  \hline
   $((C_{5}\ti C_{5})\rt C_5)\rt C_2$&  $\Diag(1,\xi_5, \xi_5^4,\xi_{5}^4,\xi_5)$, $\Diag(1,\xi_{5},\xi_5^2,\xi_{5}^{3},\xi_5^4)$,  $(12345)$, $(25)(34)$ &$X_{1}$   &$16$  & $(11,3)$ &  &new\\
  \hline
   $C_5\ti ((C_{5}\ti C_{5})\rt C_2)$&  $\Diag(1,\xi_5, \xi_5,\xi_{5}^4,\xi_5^4)$, $\Diag(1,\xi_{5},\xi_5^4,1,1)$, $\Diag(1,1,1,\xi_{5},\xi_5^4)$,  $(23)(45)$ &$X_{1}$   &$112$  & $(59,3)$ &  &\\
  \hline
  $C_5\ti ((C_{5}\ti C_{5})\rt C_3)$&  $\Diag(\xi_5, \xi_5^4,1,1,1)$, $\Diag(1,1,\xi_{5},1,\xi_5^4)$, $\Diag(1,1,1,\xi_{5},\xi_5^4)$,  $(345)$ &$X_{1}$   &$88$  & $(49,5)$ &  &\\
  \hline
   $(C_3\ti ((C_{5}\ti C_{5})\rt C_3))\rt C_2$&  $\Diag(\xi_3,\xi_3^2,\xi_5,1,\xi_{5}^4)$, $\Diag(1,1,1,\xi_{5},\xi_5^4)$, $(345)$,  $(12)(34)$ &$X_{5}$   &$56$  & $(33,5)$ &  &\\
  \hline
   $(C_5\ti ((C_{5}\ti C_{5})\rt C_2))\rt C_2$&  $\Diag(1,\xi_5, \xi_5,\xi_{5}^4,\xi_5^4)$, $\Diag(1,\xi_{5},\xi_5^4,1,1)$, $\Diag(1,1,1,\xi_{5},\xi_5^4)$,  $(23)(45)$, $(24)(35)$ &$X_{1}$   &$80$  & $(41,1)$ &  &new\\
  \hline
    $(C_{5}\ti C_{5}\ti C_5)\rt C_5$&  $\Diag(1,\xi_5,1,1,\xi_{5}^4)$, $\Diag(1,1,\xi_5,1,\xi_{5}^{4})$, $\Diag(1,1,1,\xi_5,\xi_{5}^{4})$, $(12345)$ &$X_{1}$   &$40$  & $(21,1)$ & $C_5$ &\\
  \hline
   $(C_5\ti ((C_{5}\ti C_{5})\rt C_3))\rt C_2$&  $\Diag(\xi_5, \xi_5^4,1,1,1)$, $\Diag(1,1,\xi_{5},1,\xi_5^4)$, $\Diag(1,1,1,\xi_{5},\xi_5^4)$,  $(345)$, $(12)(34)$ &$X_{1}$   &$56$  & $(33,5)$ &  &\\
  \hline
     $((C_{5}\ti C_{5}\ti C_5)\rt C_5)\rt C_2$&  $\Diag(1,\xi_5,1,1,\xi_{5}^4)$, $\Diag(1,1,\xi_5,1,\xi_{5}^{4})$, $\Diag(1,1,1,\xi_5,\xi_{5}^{4})$, $(12345)$, $(25)(34)$ &$X_{1}$   &$32$  & $(19,3)$ &  & in \cite{St12}\\
  \hline
    $((C_5\ti ((C_{5}\ti C_{5})\rt C_2))\rt C_2)\rt C_3$&  $\Diag(1,\xi_5, \xi_5,\xi_{5}^4,\xi_5^4)$, $\Diag(1,\xi_{5},\xi_5^4,1,1)$, $\Diag(1,1,1,\xi_{5},\xi_5^4)$,  $(23)(45)$, $(24)(35)$, $(345)$ &$X_{1}$   &$48$  & $(29,5)$ &  &\\
  \hline
     $(C_{5}\ti C_{5}\ti C_5)\rt A_5$&  $\Diag(1,\xi_5,1,1,\xi_{5}^4)$, $\Diag(1,1,\xi_5,1,\xi_{5}^{4})$, $\Diag(1,1,1,\xi_5,\xi_{5}^{4})$, $(12345)$, $(345)$ &$X_{1}$   &$24$  & $(15,3)$ & &new\\
  \hline
\end{longtable}
}

%\begin{remark}
%
%As in Table \ref{tablemain1}, $D$ in Table \ref{tablemain2} is defined to be $(a+b)H-C$.
%\end{remark}

\end{center}

\medskip

    \noindent
  {\bf Acknowledgement.}  \\[.2cm]
The author would like to express his deep thanks to Professor Keiji Oguiso for very valuable discussions and comments. He also would like to thank Professor Takayuki Hibi for financial support during his stay at Osaka University.

\end{document}